\documentclass[lettersize,journal]{IEEEtran}
\usepackage{amsmath,amsfonts}
\usepackage{algorithmic}
\usepackage{array}
\usepackage[caption=false,font=normalsize,labelfont=sf,textfont=sf]{subfig}
\usepackage{textcomp}
\usepackage{stfloats}
\usepackage{url}
\usepackage{verbatim}
\usepackage{graphicx}
        \usepackage{mdwmath}
        \usepackage{mdwtab}
        \usepackage{eqparbox}
\usepackage{amssymb}
\usepackage{amsthm,nccmath} 
\usepackage{cuted}
 \usepackage{tabularx,booktabs}
 \usepackage{algorithm,algorithmic}
 \usepackage{cite}
 \newcolumntype{C}{>{\centering\arraybackslash}X} 
\usepackage{lipsum}

\newtheorem{thm}{Theorem}[section]
\newtheorem{cor}[thm]{Corollary}
\newtheorem{lem}[thm]{Lemma}
\newtheorem{prop}[thm]{Proposition}

\newtheorem{ass}[thm]{Assumption}

\def\BibTeX{{\rm B\kern-.05em{\sc i\kern-.025em b}\kern-.08em
    T\kern-.1667em\lower.7ex\hbox{E}\kern-.125emX}}
\usepackage{balance}
\begin{document}
\title{Gradient-push algorithm for distributed optimization with event-triggered communications}
\author{Jimyeong Kim, Woocheol Choi
\thanks{  The authors were supported by the National Research Foundation of Korea (NRF) grants funded by the Korea government No. NRF- 2016R1A5A1008055 and No. NRF-
2021R1F1A1059671. }
\thanks{J. Kim and W. Choi are with the Department of Mathematics, Sungkyunkwan University, Suwon 16419 (e-mail:choiwc@skku.edu;  jimkim@skku.edu).}}

\markboth{ }%
{How to Use the IEEEtran \LaTeX \ Templates}

\maketitle

\begin{abstract}
Decentralized optimization problems consist of multiple agents connected by a network. The agents have each local cost function, and the goal is to minimize the sum of the functions cooperatively. It requires the agents communicate with each other, and reducing the cost for communication is desired for a communication-limited environment. In this work, we propose a gradient-push algorithm involving event-triggered communication on directed network. Each agent sends its state information to its neighbors only when the difference between the latest  sent state and the current state is larger than a threshold. The convergence of the algorithm is established under a decay and a summability condition on a stepsize and a triggering threshold. Numerical experiments are presented to support the effectiveness and the convergence results of the algorithm.
\end{abstract}

\begin{IEEEkeywords}
Decentralized optimization, event-triggering, gradient push algorithm, directed graph.
\end{IEEEkeywords}

\section{Introduction}
In recent years, distributed optimization techniques over a multi-agent network have attracted considerable attention since they play an essential role in engineering problems in distributed control \cite{BCM, CYRC}, signal processing \cite{Boyd Gossip, QT} and the machine learning problems \cite{BCN, FCG, RB}. 
In distributed optimization, many agents have their own local cost and try to find a minimizer of the sum of those local cost functions in a collaborative way that each agent only uses the information from its neighboring agents where the neighborhood structure is depicted as a graph, often undirected or directed. 

There has been a significant interest in consensus-based distributed gradient method. One fundamental work is \cite{Ned DGD} which developed the distributed gradient descent on undirected graph.  This algorithm consists of local gradient step and consensus step based on communication between neighboring agents. The convergence property of the algorithm has been studied in the works \cite{Ned DGD, R Ned V - stochastic, I, YLY}. There are also various distributed algorithms containing the distributed dual averaging method \cite{DAW}, consensus-based dual decomposition \cite{FMGP, SJ}, and the alternating direction method of multipliers (ADMM) based algorithms \cite{MJ, SLYW}. 
These algorithms work with doubly-stochastic matrix associated with the undirected graph. 

The gradient-push algorithm was  introduced in \cite{Ned Push sum} to solve the distributed optimization for  directed graph which utilizes push-sum algorithms \cite{KDG, TLR2}. The communication of this algorithm is represented by column stochastic matrix, which requires each agent to know its out-degree at each time, without having the information of the number of agents. This algorithm has influenced a significant impact of later works. The work  \cite{NO1} studied
the  algorithm with gradient having a noise. The time varying distributed optimization was also considered \cite{AGL} using the gradient-push algorithm. Recently, stochastic gradient-push algorithm was designed for large scale deep learning problem \cite{ALBR}. This work was also extended in \cite{TMHP} further to quantized communication settings. We also refer to \cite{AR, ZY} where the authors studied the asynchronous version of the gradient-push algorithm. 

Regardless of the types of graphs, these distributed algorithms require each agent to communicate with their neighbors at every iteration, which leads to a overhead in restricted environments. Power consumption by communication may become more significant than that by computation of control inputs or optimization algorithms \cite{SHCA}. Recently,  the event-triggering approach has appeared as a promising paradigm to reduce the communication load in distributed systems. In the distributed detection problem over sensor network \cite{AVJ, AMM}, each sensor censors its local data and send the updated data to the fusion center only when the data is informative. For distributed control problems, agents send their coordinate information only when a triggering condition is satisfied \cite{MT, DFJ}.

For the distributed optimization problems, recent works \cite{KHT, LL, MUA, WYLY, ZC} developed distributed optimization algorithms with event-triggered communication to overcome the communication overhead of distributed systems. 
Lu-Li \cite{LL} designed the distributed gradient descent with event-triggered communication
for the distributed optimization on the whole space, and it was further studied in
Li-Mu \cite{LM} to establish a convergence rate.
 For the distributed optimization on bounded
domain, Kajiyama et al \cite{KHT} designed the projected distributed gradient descent with
event-triggered communication. Liu et al \cite{LLSX} extended the work to the case with constant
step-size. Cao-Basar \cite{CB} studied the online distributed problem using the distributed
event-triggered gradient method. 
 In these algorithms, each agent sends its state only when the difference between the current state and the latest sent state is larger than a threshold, 
therefore reducing possible unnecessary network utilization.

The consensus-based distributed optimization algorithms with event-triggering communication mentioned above have been proposed for the undirected graph. In this work, we are interested in 
developing a distributed optimization on directed graph involving the even-triggered communication. Precisely we propose the gradient-push algorithm incorporating the event-triggered communication. 
In the proposed algorithm, each agent only sends its state information when the differences between the latest sent states and the current states are larger than a triggering thresholds. We prove that the algorithm solves the distributed optimization under suitable decays and summability conditions on the stepsize and the triggering thresholds. The numerical experiments are given for the proposed algorithm, supporting the theoretical results.

The rest of the paper is organized as follows. In Section 2, we state the problem and introduce the algorithm with its convergence results. Section 3 is devoted to provide a consensus estimate, which is essentially used in Section 4 to prove the convergence results. In section 5, we present numerical results of the proposed algorithm.

Before ending this section, we state several notations used in this paper.   For a matrix $A \in \mathbb{R}^{n \times m}$, $a_{ij}$ or $[A]_{ij}$ denotes the $(i,j)$th entry of A. For a vector $x\in\mathbb{R}^d$, $\|x\|=\sqrt{x^Tx}$ denotes the standard Eculidean norm. In addition, for $X \in \mathbb{R}^{m\times d}$ given by $X=[x_1; x_2; \cdots ; x_m]$ with row vector $x_k \in \mathbb{R}^{d}$, we define the mixed norm $\|X\|_1$ by $\|X\|_1 = \sum_{k=1}^m \|x_k\|$ nad the maximum norm $\|X\|_{\infty} = \max_{1\leq k\leq m} \|x_k\|$. Also we use $\bar{x}$ to denote $\bar{x} = \frac{1}{m}\sum^{m}_{k=1} x_k$.

\section{Problem, algorithm, and main results}
\subsection{Problem statement} We consider the distributed optimization problem, which consists of $m$ agents connected by a network that collaboratively minimize a global cost function given by the sum of local private cost functions. Formally, the problem is described by

\begin{equation}\label{problem}
\min_{x\in\mathbb{R}^d}f(x)=\sum^m_{i=1}f_i(x),
\end{equation}  
where $f_i:\mathbb{R}^d\rightarrow\mathbb{R}$  is a local convex cost function only known to agent $i\in \mathcal{V}=\{1,2,\cdots m\}$. We let $f^*$ be the optimal value of problem \eqref{problem} and denote by $X^*$ the set of optimal solutions, i.e., 
$$
X^*=\{x \in \mathbb{R}^d~:~ f(x) = f^*\},
$$
which is assumed to be nonempty.   We make the following standard assumption on the local cost functions.

\begin{ass}\label{GB}
For each $i \in \{1,\cdots, m\}$, there exists $D_i>0$ such that
\begin{equation}\label{eq-2-1}
\|\nabla f_i(x)\| \leq D_i \quad \forall x \in \mathbb{R}^d.
\end{equation}
We set $D = \max_{1\leq i \leq m} D_i$.
\end{ass}

The communication pattern among agents in \eqref{problem} at each time $t \in \mathbb{N}\cup \{0\}$ is characterized by a directed graph $\mathcal{G}(t)=(\mathcal{V},\mathcal{E}(t))$, where each node in $\mathcal{V}$ represents each agent, and each directed edge $(i,j)\in \mathcal{E}(t)$ means that $i$ can send messages to $j$. In this work, we consider a sequence of graphs $\{\mathcal{G}(t)\}_{t\geq0}$ satisfying the following assumption.
\begin{ass}\label{graph}
The sequence of graph $\{\mathcal{G}(t)\}_{t\geq0}$ is uniformly strongly connected, i.e., there exists a value $B \in \mathbb{N}$ such that  the graph with edge set $\cup_{i=kB}^{(k+1)B -1} \mathcal{E}(i)$ is strongly connected for any $k \geq 0$. 
\end{ass}

We define in-neighbors and out-neighbors of node $i$, respectively, as $N_i^{\text{in}}(t)=\{j|(j,i)\in\mathcal{E}(t)\}\cup \{i\}$ and $N_i^{\text{out}}(t)=\{j|(i,j)\in\mathcal{E}(t)\}\cup \{i\}$. Also the out-degree of node $i$ is defined as $d_i^{\text{out}}(t)=|N_i^{\text{out}}(t)|$. Define the matrix A(t) such that
$[A(t)]_{ij} = a_{ij}(t)$, where
$$
a_{ij}(t)=\begin{cases}
    1/d_j^{\text{out}}(t),\ \text{if $j\in N_i^{\text{in}}(t)$},\\
    0, \ \text{otherwise}.
\end{cases}
$$     
The matrix $A(t) \in \mathbb{R}^{m\times m}$  is column stochastic and we recall some useful properties of this matrix from \cite[Corollary 2]{Ned Push sum}.

\begin{lem}[\cite{Ned Push sum}, $Corollary$ 2]\label{lemma2}
Suppose that the graph sequence $\{G(t)\}$ is uniformly strongly connected. Then, the following statements are valid.
\begin{enumerate}
\item For each integer $s\geq 0$, there is a stochastic vector $\phi(s)$ such that for all $i,j$ and $t\geq s$
\begin{equation}\label{MatrixEstimate}
|[A(t:s)]_{ij}-\phi_i(t)| \leq C_0  \lambda^{t-s}
\end{equation}
for some values $C_0 \geq 1$ and $\lambda\in(0,1)$ depending on the graph sequence.
\item The following inequality holds.
\begin{equation}\label{Q}
Q := \inf_{t=0,1,\cdots}{\min_{1\leq i \leq m}{[A(t:0)\mathbf{1}]_i}} \geq 1/n^{nB}.
\end{equation}
\end{enumerate}
 Here we denote by $A(t:s)$ the matrix given as
\begin{equation*}
A(t:s) = A(t) A(t-s) \cdots A(s) \quad \textrm{for all}~ t \geq s \geq 0.
\end{equation*}
\end{lem}

 
     \begin{algorithm}
    \begin{algorithmic}[1] 
    \caption{Distributed Event-Triggered gradient-push algorithm}\label{algo}
    \REQUIRE Initialize $x_i(0)$ arbitraily, $y_i(0)=1$ for all $i\in\{1,\cdots,m\}$. Set $\hat{x}_i(0)=x_i(0)$
            \FOR{$t=0,1,\cdots$,}
                \STATE Compute the new action as  
                 \begin{align}
            \widehat{w}_i(t+1) &= \sum^m_{j=1} a_{ij}(t) \hat{x}_j(t), \\
             y_i(t+1) &= \sum^m_{j=1} a_{ij}(t) \hat{y}_j(t),\label{hatyy} &\\
            \widehat{z}_i(t+1)&=\frac{\widehat{w}_i(t+1)}{y_i(t+1)}, \label{hatzz}
            \end{align}
            \begin{equation}\label{update}
            x_i(t+1) = \widehat{w}_i(t+1) - \alpha(t+1) \nabla f_i(\widehat{z}_i(t+1)).
            \end{equation} 
            \IF{$\|x_i(t+1)-\hat{x}_i(t)\|\geq \tau(t+1)$}
                \STATE Send $x_i(t+1)$ to neighbors $\mathcal{N}_i$. Set $\hat{x}_i(t+1) = x_i(t+1)$
                \ELSE
                \STATE Do not send and set $\hat{x}_i(t+1) = \hat{x}_i(t)$
                \ENDIF
                \IF{$|y_i(t+1)-\hat{y}_i(t)|\geq \zeta(t+1)$}
                \STATE Send $y_i(t+1)$ to neighbors $\mathcal{N}_i$. Set $\hat{y}_i(t+1) = y_i(t+1)$
                \ELSE
                \STATE Do not send and set $\hat{y}_i(t+1) = \hat{y}_i(t)$
                \ENDIF
            \ENDFOR
    \end{algorithmic}
    \end{algorithm}

In addition, we consider the following assumptions on the stepsize and the thresholds for the trigger conditions.
\begin{ass}\label{stepsize}
The sequence of stepsize $\{\alpha(t+1)\}_{t\geq 0}$ is monotonically non-increasing and satisfies
$$
\sum^\infty_{t=0}\alpha(t+1)=\infty, \ \sum^\infty_{t=0}\alpha(t+1)^2<\infty. 
$$  
\end{ass}
\begin{ass}\label{event-triggered}
The sequence of event-triggering thresholds $\{\tau(t)\}_{t\geq 0}$ is monotonically non-increasing and satisfies
$$
\sum^\infty_{t=1}\tau(t)<\infty.
$$   
\end{ass}
\begin{ass}\label{y-trigger}
The sequence of event-triggering thresholds $\{\zeta(t)\}_{t\geq 0}$ is monotonically non-increasing and satisfies
\begin{equation*}
\sum^\infty_{t=1} t^{3/2}\zeta(t) < \infty, \quad \sum^\infty_{t=1} \zeta(t) <1.
\end{equation*}
\end{ass}
Note that $\sum^\infty_{t=1} t^{3/2}\zeta(t) < \infty$ implies that there exists a finite $M$ such that $\sum^\infty_{t=0} \zeta(t) = M$. If we set a new sequence $\{\tilde{\zeta}(t)\}_{t\in \mathbb{N}}$ by $\tilde{\zeta}(t) = \zeta(t)/(M+1)$, then it satisfy $\sum^\infty_{t=1} \tilde{\zeta}(t)<1$. Hence if we have a sequence $\{\zeta(t)\}_{t\in \mathbb{N}}$ satisfying $\sum^\infty_{t=1} t^{3/2}\zeta(t) < \infty$, then we may divide the sequence by a positive constant to satisfy Assumption \ref{y-trigger}. One example of the sequence that satisfies Assumption \ref{y-trigger} is $\zeta(t) = \frac{1}{3t^3}$.
For $T \geq 0$ we set the following variables 
\begin{equation}\label{eq-2-8}
E_{\tau}(T) = \sum_{t=1}^T \tau (t), \quad E_{\tau,2} (T) = \sum_{t=1}^T \tau(t)^2, \quad E_{\tau}= \sum_{t=1}^{\infty} \tau (t)
\end{equation}
and
\begin{equation}\label{eq-2-8}
F_{\zeta}(T) = \sum_{t=1}^T \zeta (t), \quad F_{\zeta}= \sum_{t=1}^{\infty} \zeta (t), \quad  F_{\zeta_{3/2}} = \sum_{t=1}^\infty t^{3/2}\zeta(t).
\end{equation}
These variables will appear in the statements and proofs of  the convergence result for Algorithm 1.
\
Before finishing this subsection, we define the following  constant
\begin{equation}\label{delta}
\delta:= \min_{1\leq i\leq m }\inf_{t\in\mathbb{N}} y_i(t).
\end{equation}
whose positivity is proved in Lemma \ref{lowerbound} under the Assumption \ref{y-trigger}.

\subsection{Main results}
Our first result establishes the convergence of $\hat{z}_i(t)$ to the optimal solutions for an arbitrary stepsize $\alpha(t)$ satisfying Assumption \ref{stepsize}, and event-triggering thresholds $\tau(t)$ and $\zeta(t)$ satisfying Assumption \ref{event-triggered} and \ref{y-trigger}. 
\begin{thm}\label{maintheorem}
Suppose that Assumptions \ref{GB} - \ref{y-trigger} hold. Then the sequence $\{\hat{z}_i (t)\}_{t \in \mathbb{N}}$ for $1\leq i \leq n$ of the Algorithm \ref{algo} satisfies the following property:
$$
\lim_{t\rightarrow \infty} \hat{z}_i(t) = x^* \ \text{for all $i$ and for some $x^*\in X^*$}.
$$
\end{thm}
Next we consider the Algorithm \ref{algo} with specific stepsize $\alpha(t)=1/\sqrt{t}$. This stepsize does not satisfy Assumption \ref{stepsize}, but we may obtain an explicit convergence rate as in the following result. Before stating the result, we give some notations which are used throughout the paper. We define the constants $m_{\zeta}$ and $B_{\zeta}$ by
\begin{equation}\label{mzeta}
 m_{\zeta} = 1_m^T y(0) + \sum_{s=1}^{\infty} 1_m^T \theta (s) = m+ \sum_{s=1}^{\infty} 1_m^T \theta (s)
 \end{equation}
 and
 \begin{equation}\label{bzeta}
 B_{\zeta} = \frac{m_{\zeta}}{m}.
 \end{equation}
These constants are well-defined if $\{\zeta(s)\}_{s \geq 0}$ is summable since we have the inequality $|1_m^T \theta (s)| \leq m \zeta (s)$ from the triggering condition.
\begin{thm}\label{case2_main}
Suppose that Assumption \ref{GB}, \ref{graph}, \ref{event-triggered} and \ref{y-trigger} hold. Let $\alpha(t)=\frac{1}{\sqrt{t}}$ for $t \geq 0$. Define $H(-1)=1$ and $H(t) : =\prod_{k=0}^t (1+\tau(k))$. Moreover, suppose that every node $i$ maintains the variable $\tilde{z}_i(t)\in \mathbb{R}^d$ initialized at time $t=0$ with $\tilde{z}_i(0)\in\mathbb{R}^d$ and updated by
$$
\tilde{z}_i(t+1) = \frac{\frac{\alpha(t+1)}{H(t)}\hat{z}_i(t+1)+S(t)\tilde{z}_i(t)}{S(t+1)},
$$
where $S(0)=0$ and $S(t)=\sum^{t-1}_{k=0}\frac{\alpha(k+1)}{H(k)}$ for $t\geq 1$. Then we have for each $T\geq 0$ and $i=1,\cdots,m$, the following estimate 
\begin{equation*}
\begin{split}
&f(\tilde{z}_i(T+1))-f(x^*) 
\\
&{\leq \frac{me^{E_{\tau}}}{2\sqrt{T+1}}J_1 (T) + \frac{3mDe^{E_{\tau}}}{\delta \sqrt{T+1}} J_2 (T) + \frac{3mDe^{E_\tau}}{ \delta\sqrt{T+1}} J_3 (T),}
\end{split}
\end{equation*}
where
\begin{align*}
&\medmath{ J_1 (T)= \frac{\|\bar{x}(0)-x\|^2}{B_{\zeta}}+\bigg[2D^2 \Big(1+ \ln{(T+1)}\Big)+2E_{\tau,2}(T)+E_{\tau}(T)\Big)\bigg]B_{\zeta}}
\\
&J_2 (T)= \medmath{\bigg(\frac{C_0 }{ (1-\lambda)} \bigg)\|x(0)\|_1  + \frac{4mC_0  E_{\tau}(T)}{ (1-\lambda)}  +\bigg(  \frac{C_0mD}{ (1-\lambda)}\bigg)(1+\ln(T))}
\\
&J_3 (T)=\medmath{ \sum_{t=0}^T K(t) \alpha (t+1) \Big[ \|x(0)\|_1 + \sum_{s=0}^{t-1} (\alpha (s+1) D + \tau (s))\Big],}
\end{align*}
and $x^*\in X^*$. In addition, the right hand side is bounded by $O(\log (T+1)/ \sqrt{T+1})$. Here $K(t) = \beta(t)/m_{\zeta}$ denotes a bound for asymptotic behavior of $y_i (t)$ stated in \mbox{Lemma \ref{lem-3-4}.}
\end{thm}

    \section{Properties of the sequence $\{y_i(t)\}_{t\in\mathbb{N}}$ and Disagreement in Agent Estimates}\label{sec3}
\subsection{Properties of the sequence $\{y_i(t)\}_{t\in\mathbb{N}}$}
A convergence property of $\{y_i(t)\}_{t\in\mathbb{N}}$ and their positive uniform lower bound are key points in proving our main results. Let us first look at the case without event-triggering $(\zeta(t)=0)$, which means all in-neighbors of agent $i$ share the $y_i(t)$ with this agent for every time step. In this case, since $y(t)=\hat{y}(t)$ for all $t\in \mathbb{N}$, it holds that
\begin{equation}\label{yyy}
y(t) = A(t-1:0)\mathbf{1}_m
\end{equation}
by \eqref {hatyy} in Algorithm \ref{algo}. Hence we can directly show that $y(t)$ converges to $\phi(t)$ and  has a uniform lower bound $Q$  using Lemma \ref{lemma2}.
In the event-triggered case, $y(t)$ can be written as
\begin{equation}\label{yyyy}
y(t) = A(t-1:0) \mathbf{1}_m + \sum^{t-1}_{s=1}A(t-1:s)\theta(s),
\end{equation}
where $\theta(s) = \hat{y}(s) - y(s)$.
Therefore the convergence and uniform lower boundedness property may not hold due to the additional term 
$$
\sum^{t-1}_{s=1}A(t-1:s)\theta(s).
$$
The following lemmas shows that $y(t)$ has a positive uniform lower bound $\delta$ and converges to $m_\zeta\phi(t)$ instead of $\phi(t)$ under Assumption \ref{y-trigger}.     
\begin{lem}\label{lem-3-4}
Suppose that Assumption \ref{y-trigger} holds. Then we have
\begin{equation}\label{m_lowerbound}
m_\zeta \geq (1-F_\zeta) m
\end{equation}
and the following estimate holds: 
\begin{equation}\label{y_consensus}
\Big\| y(t+1) -  m_{\zeta}\phi(t)\Big\|_{\infty} \leq \beta(t) \quad \forall~t \in \mathbb{N},
\end{equation}
where 
\begin{equation*}
\beta (t) =m\bigg(\Big(F_\zeta-F_{\zeta}(t)\Big) +C_0  \lambda^{t}+C_0\lambda^{t/2} F_\zeta(t) + \frac{\zeta([t/2]+1)}{1-\lambda}\bigg).
\end{equation*}
In addition, we have 
\begin{equation*}
\lim_{t \rightarrow \infty} t^{3/2} \beta (t) =0.
\end{equation*}
 
\end{lem}

\begin{proof}
Observe that $|\theta_i (s)| =|\hat{y}_i (s) - y_i (s)| \leq \zeta (s)$ for $s \geq 1$ by the event-triggering condition, and so
\begin{equation*}
\Big|\sum_{s=1}^{\infty}1_m^T \theta (s) \Big|\leq m \sum_{s=1}^{\infty} \zeta(s) = m F_\zeta.
\end{equation*}
Using this in \eqref{mzeta}, we get
\begin{equation*}
m_{\zeta} = m + \sum_{s=1}^{\infty} 1_m^T \theta(s) \geq (1-F_\zeta)m,
\end{equation*}
which proves \eqref{m_lowerbound}. 

Next we prove \eqref{y_consensus}. For each $s \geq 0$, by definition we have
\begin{equation*}
y(s+1) = A(s) \hat{y}(s) = A(s) (y(s) + \theta(s)),
\end{equation*}
where we have set $\theta (0)=0$. Using this iteratively gives the following formula
\begin{equation*}
\begin{split}
&y(t+1)
\\
& = A(t:0)y(0) + \sum_{s=1}^{t} A(t:s) \theta(s)
\\
& = \phi (t) \Big[ \mathbf{1}_m^T y(0) + \sum_{s=1}^{t} \mathbf{1}_m^T \theta (s) \Big] 
\\
&\quad + (A(t:0) - \phi (t) 1_m^T) y(0) + \sum_{s=1}^{t} \Big[ (A(t:s) - \phi (t) \mathbf{1}_m^T) \theta (s) \Big].
\end{split}
\end{equation*}
Since $|\theta_j (s)| \leq \zeta(s)$, we find
\begin{equation*}
\begin{split}
\sum_{s=t+1}^{\infty} |\theta_j (s)| &\leq \sum_{s=t+1}^{\infty}\zeta(s) = F_\zeta-F_{\zeta}(t).
\end{split}
\end{equation*}
Using the above inequality, we obtain
\begin{equation}\label{eq-3-310}
\Big|m_{\zeta} - 1_m^T y(0) - \sum_{s=0}^{t}1_m^T \theta (s) \Big|= \Big| \sum_{s=t+1}^{\infty} 1_m^T\theta (s) \Big| \leq \Big(F_\zeta-F_{\zeta}(t)\Big)m.
\end{equation}
Hence we have
\begin{equation}\label{eq-3-340}
\begin{split}
&\Big\| y(t+1) - m_\zeta \phi(t)\Big\|_\infty
\\
 &\leq  \Big(F_\zeta-F_{\zeta}(t)\Big)m + \Big\|(A(t:0) - \phi (t) 1_m^T) y(0)\Big\|_\infty\\
&\quad + \Big\|\sum_{s=1}^{t} \Big[ (A(t:s) - \phi (t) 1_m^T) \theta (s) \Big]\Big\|_{\infty}.
\end{split}
\end{equation}
Now we estimate the second and third terms in the right hand side of \eqref{eq-3-340}. Using \eqref{MatrixEstimate} we have
\begin{equation}\label{eq-3-320}
\Big\|(A(t:0) - \phi (t) 1_m^T) y(0)\Big\|_\infty \leq m C_0  \lambda^{t} 
\end{equation}
and
\begin{equation}\label{eq-3-330}
\begin{split}
&\Big\|\sum_{s=1}^{t} \Big[ (A(t:s) - \phi (t) 1_m^T) \theta (s) \Big]\Big\|_{\infty} 
\\
&\leq mC_0  \sum_{s=1}^{t}\lambda^{t-s} \zeta (s) 
\\
&\leq mC_0 \bigg(\sum^{[t/2]}_{s=1} \lambda^{t-s}\zeta(s) + \sum^t_{s=[t/2] +1} \lambda^{t-s} \zeta(s)\bigg)\\
&\leq mC_0\bigg( \lambda^{t/2} F_\zeta(t) + \zeta([t/2]+1) \sum^t_{s=[t/2] +1} \lambda^{t-s}\bigg)\\
& \leq mC_0\bigg(\lambda^{t/2} F_\zeta(t) + \frac{\zeta([t/2]+1)}{1-\lambda} \bigg).
\end{split}
\end{equation}
Putting the estimates \eqref{eq-3-320} and \eqref{eq-3-330} in \eqref{eq-3-340}, we get
\begin{equation*}
\begin{split}
&\Big\| y(t+1) - \phi(t) m_{\zeta}\Big\|_{\infty}
\\
& \leq  m\bigg(\Big(F_\zeta-F_{\zeta}(t)\Big) +C_0  \lambda^{t}+C_0\lambda^{t/2} F_\zeta(t) + \frac{\zeta([t/2]+1)}{1-\lambda}\bigg).
\end{split}
\end{equation*}
  This proves the second assertion of the lemma.
  
Now we shall show that $\lim_{t \rightarrow \infty} t^{3/2} \beta (t) = 0$. Since $\lambda \in (0,1)$, it suffices to show that 
\begin{equation*}
\lim_{t \rightarrow \infty} t^{3/2} (F_{\zeta} - F_{\zeta}(t) + \zeta([t/2]))=0.
\end{equation*}
This fact follows directly from  the fact that $\sum_{t=0}^{\infty} t^{3/2} \zeta (t) < \infty$ and the following inequality
\begin{equation*}
t^{3/2} (F_{\zeta} -F_{\zeta}(t)) = t^{3/2} \sum_{s=t+1}^{\infty}\zeta (s) \leq \sum_{s=t+1}^{\infty}s^{3/2} \zeta (s).
\end{equation*}
The proof is done.
\end{proof}

\begin{lem}\label{lowerbound}
Suppose that Assumption \ref{y-trigger} holds. Then  the value $\delta \in \mathbb{R}$ defined in \eqref{delta} is positive.
\end{lem}
\begin{proof}
Note that from \eqref{MatrixEstimate} and \eqref{Q}, we have
\begin{equation*}
\begin{split}
m \phi_i (t) & = \sum_{j=1}^m [A(t:0)]_{ij} + \sum_{j=1}^m \Big( \phi_i (t) - [A(t:0)]_{ij}\Big)
\\
&\geq   Q - mC_0  \lambda^t.
\end{split}
\end{equation*}
Using the above inequality and Lemma \ref{lem-3-4}, we deduce for each $1\leq i \leq m$ the following estimate 
\begin{equation*}
\begin{split}
y_i (t+1) &\geq m_{\zeta}\phi_i (t)  -\beta(t)
\\
& \geq (m_{\zeta}/m)Q - m_{\zeta} C_0 \lambda^t - \beta (t).
\end{split}
\end{equation*}

Since $\beta(t)$ converges to zero as $t$ goes to infinity and $\lambda \in (0,1)$, there exists a time $T\in \mathbb{N}$ and a constant $\tilde{\delta}>0$ such that for any $t\geq T$, 
\begin{equation}\label{ddd}
y_i(t+1) \geq  \tilde{\delta}.
\end{equation}
Note that by Assumption \ref{graph}, each matrix $A(t)$ has no zero row. This fact, together with the definition of $\hat{y}$ and \eqref{hatyy}, for any $t \in \mathbb{N}$ we have 
\begin{equation}\label{eq-3-35}
\min_{1\leq i\leq m} y_i (t) >0.
\end{equation}
Therefore, combining \eqref{ddd} with \eqref{eq-3-35}, we conclude that  $\delta$ defined in \eqref{delta} satisfies 
\begin{equation*}
\delta \geq \min_{1\leq i \leq m}{\{y_i(0), y_i(1), \cdots, y_i(T), \tilde{\delta}\}}>0.
\end{equation*}
The proof is done.
\end{proof}

\subsection{Disagreement in Agent Estimates}
In this subsection, we derive a bound of the disagreement in agent estimates $\{\hat{z}_i (t)\}_{i=1}^m$ that will be used in the proofs of the main theorems. In the case without event-triggering $(\tau(t)=\zeta(t)=0)$, the paper \cite{Ned Push sum} proved that $\|\hat{z}_i(t+1) -\bar{x}(t)\|$ converges to zero for the stepsize satisfying Assumption \ref{stepsize} as $t$ goes to infinity. For the event-triggered case, the following proposition shows that the values $\{\hat{z}_i(t)\}_{i=1}^m$ approach $B_{\zeta}\bar{x}(t)$ instead of $\bar{x}(t)$ as $t$ goes to infinity due to the effect of the threshold $\zeta(t)$ for the triggering condition of $\{y_i(t)\}_{i=1}^m$. 
\begin{prop}\label{mean}
For any $t \geq 1$ we have
\begin{equation*}
\begin{split}
&\|\hat{z}_i(t+1)-B_{\zeta}\bar{x}(t)\|\\
&\leq \frac{1}{\delta}\bigg(\Big( C_0  \lambda^t +K(t)\Big) \|x(0)\|_1
\\
&\quad + m\sum_{s=0}^{t-1} \Big[ C_0 \lambda^{t-s-1} +K(t)\Big] \Big( \alpha (s+1) D +  \tau (s)\Big)\bigg) \\
&\qquad +\frac{1}{\delta}\bigg(  d_i(t)\tau(t)\bigg),
\end{split}
\end{equation*}
where $K(t) = \beta(t)/m_{\zeta}$ and for $t=0$ we have
\begin{align*}
\|\hat{z}_i(1)-\bar{x}(0)\|\leq \frac{2C_0 }{\delta}  \|x(0)\|,
\end{align*}
where the constant $\delta>0$ satisfies $y_i(t)>\delta$ for all $t>0$. 
\end{prop}
To prove Proposition \ref{mean}, we consider a variable  $w_i(t+1) \in \mathbb{R}^d$ which is a companion to the variable $\hat{w}_i (t+1) \in \mathbb{R}^n$ defined as 
\begin{equation}\label{eq-3-10}
w_i(t+1) = \sum^m_{j=1} a_{ij} (t) x_j(t),
\end{equation}
and their difference 
\begin{equation}\label{error}
e_i(t+1) = \hat{w}_i(t+1) - w_i(t+1).
\end{equation}
Then we may rewrite the gradient step \eqref{update} as
\begin{equation}\label{PEalgorithm}
\begin{split}
x_i(t+1)& = w_i(t+1) - \alpha(t+1) \nabla f_i(\hat{z}_i(t+1)) + e_i(t+1).
\end{split}
\end{equation} Summing up \eqref{PEalgorithm} for $1\leq i \leq m$ and using that $A(t)$ is column-stochastic, we have
\begin{equation}\label{eq-2-1a}
\bar{x}(t+1) = \bar{x}(t) - \frac{\alpha(t+1)}{m}\sum^m_{i=1}\nabla f_i(\hat{z}_i(t+1)) + \frac{1}{m}\sum^m_{i=1}e_i(t+1).
\end{equation} 
Now we  find a bound of $e_i(t+1)$ which is the difference  between $w_i(t+1)$ and $\hat{w}_i(t+1)$ associated to  the event-triggering $\tau(t)$.  
\begin{lem}
The quantity $e_i(t+1)$ defined in \eqref{error} satisfies
\begin{equation}\label{error_estimate}
\|e_i(t+1)\| \leq d_i (t) \tau(t),
\end{equation}
where $d_i (t)= \sum_{j=1}^m a_{ij}(t)$. In addition, we have
\begin{equation}\label{error_estimate_2}
\sum_{i=1}^m \|e_i (t+1) \| \leq m \tau (t).
\end{equation}
\end{lem}

\begin{proof}
By using the triggering condition, we have
\begin{align*}
\|e_i(t+1)\| &\leq \|\hat{w}_i(t+1) -w_i(t+1)\| \\
&\leq \bigg\| \sum^m_{j=1} a_{ij}(t)(\hat{x}_j(t)-x_j(t))  \bigg\| \\
&\leq \sum^m_{j=1}a_{ij}(t)  \|\hat{x}_j(t)-x_j(t)\| \leq d_i (t) \tau(t),
\end{align*}
which proves \eqref{error_estimate}. Summing this over $1\leq i \leq m$ and using that $A(t)$ is column stochastic, we find
\begin{equation*}
\begin{split}
\sum_{i=1}^m \|e_i (t+1) \| &\leq \sum_{i=1}^m \sum_{j=1}^m \Big(a_{ij} (t) \tau (t)\Big) 
\\
&= \sum_{j=1}^m \Big( \sum_{i=1}^m a_{ij}(t)\Big) \tau (t) = m \tau (t).
\end{split}
\end{equation*} 
The proof is finished.
\end{proof}

Now we are ready to prove Proposition \ref{mean}.
\begin{proof}[Proof of Proposition \ref{mean}]
We regard $x_k (t)$ as a row vector in $\mathbb{R}^{1 \times d}$ and define the variables $x(t)\in \mathbb{R}^{m \times d}$, $\nabla f(\hat{z}(t))\in \mathbb{R}^{m \times d}$, and $e (t) \in \mathbb{R}^{m \times d}$ as
\begin{equation*}
\medmath{
x(t) = \begin{pmatrix} x_1 (t) \\ \vdots \\ x_m (t) \end{pmatrix},~ \nabla f(\hat{z}(t)) = \begin{pmatrix} \nabla f_1 (\hat{z}_1 (t)) \\ \vdots \\ \nabla f_m (\hat{z}_m (t))\end{pmatrix},~ e (t) = \begin{pmatrix} e_1 (t) \\ \vdots \\ e_m (t) \end{pmatrix}.}
\end{equation*}
Note that by \eqref{error} and \eqref{hatzz}, we have 
$$
\hat{z}_i(t+1) = \frac{w_i(t+1) +e_i(t+1)}{y_i(t+1)}.
$$ 
Also we see from definition \eqref{bzeta} that
$$
B_{\zeta}\bar{x}(t) =\frac{m\bar{x}(t)}{m_\zeta}=\frac{\mathbf{1}_m^Tx(t)}{m_\zeta}.
$$
Using these formulas and \eqref{eq-3-10} we have
\begin{equation}\label{eq-3-51}
\begin{split}
&\hat{z}_i (t+1) -B_{\zeta} \bar{x}(t)
\\
& = \frac{w_i (t+1) + e_i (t+1)}{y_i (t+1)} -  \frac{\mathbf{1}_m^T x(t)}{m_{\zeta}}
\\
& = \frac{1}{y_i (t+1)} \bigg( [A(t)x(t)]_i - y_{i} (t+1) \frac{1_m^T x(t)}{m_{\zeta}}\bigg) + \frac{e_i (t+1)}{y_i (t+1)}.
\end{split}
\end{equation}
To estimate the first term on the right hand side of the last equality,  we rewrite \eqref{PEalgorithm} as
\begin{equation*}
{x}(t+1) = A(t) {x}(t) - \alpha(t+1) \nabla f(\hat{z}(t+1)) + e(t+1).
\end{equation*}
Using this formula recursively, for $t \geq 1$ we  have
\begin{equation}\label{Ax}
A (t) {x}(t)=A(t:0) {x}(0)-\sum^{t-1}_{s=0}A(t:s+1)\varepsilon(s),
\end{equation}
where we have let 
$$\varepsilon(s)=\alpha(s+1)\nabla f(\hat{z}(s+1))-e(s+1).$$
Using Assumption \ref{GB} and \eqref{error_estimate_2} we have the following bound 
\begin{equation}\label{eq100}
\|\varepsilon(s)\|_1 \leq m\Big(\alpha(s+1)D +\tau (s)\Big).
\end{equation}
Since $A(t)$ is column stochastic we have $1_m^T A(t) = 1_m^T$, and combine this with \eqref{Ax} to have
\begin{equation}\label{barx}
\mathbf{1}_m^T x(t) = \mathbf{1}_m^T x(0) -\sum^{t-1}_{s=0} \mathbf{1}_m^T \varepsilon(s). 
\end{equation}
Combining \eqref{Ax} and \eqref{barx} yields
\begin{equation}\label{important}
\begin{split}
A(t) x(t)&= \phi(t)\mathbf{1}_m^T x(t)+\Big(A(t:0)-\phi (t)\mathbf{1}_m^T\Big)x(0) 
\\
&\quad -\sum^{t-1}_{s=0}\Big(A(t:s+1)-\phi (t)\mathbf{1}_m^T\Big)\varepsilon (s),
\end{split}
\end{equation}
where $\phi (t)$ is the stochastic vector satisfying \eqref{MatrixEstimate}. By Lemma \ref{lem-3-4}, for $y(t):= (y_1 (t), \cdots, y_m (t))^T \in \mathbb{R}^{m \times 1}$ we have
\begin{equation*}
y(t+1) = m_{\zeta}\phi(t) + r(t),
\end{equation*}
where $r(t)$ satisfies $\|r(t)\|_{\infty} \leq \beta(t)$.
Combining this with \eqref{important}, we obtain
\begin{equation*}
\begin{split}
 & [A(t)x(t)]_i - y_i (t+1) \frac{1_m^T x(t)}{m_{\zeta}}
\\
& = \phi_i (t) 1_m^T x(t) + [(A(t:0)- \phi (t)1_m^T) x(0)]_i
\\
&\qquad - \sum_{s=0}^{t-1} \Big[(A(t:s+1) - \phi (t) 1_m^T) \varepsilon (s)\Big]_i
\\
&\qquad -[m_{\zeta} \phi_i (t) +r_i (t)] \frac{1_m^T x(t)}{m_{\zeta}}
\\
& =  [(A(t:0)- \phi (t)\mathbf{1}_m^T) x(0)]_i 
\\
&\qquad- \sum_{s=0}^{t-1} \Big[(A(t:s+1) - \phi (t) \mathbf{1}_m^T)\varepsilon (s)\Big]_i - r_i (t) \frac{1_m^T x(t)}{m_{\zeta}}.
\end{split}
\end{equation*}
By applying \eqref{MatrixEstimate} here, we deduce 
\begin{equation}\label{eq-3-50}
\begin{split}
&\Bigg\| [A(t)x(t)]_i -y_i (t+1) \frac{1_m^T x(t)}{m} \Bigg\|
\\
&\leq C_0   \lambda^t \|x(0)\|_1 + \sum_{s=0}^{t-1} C_0  \lambda^{t-s-1} \|\varepsilon (s)\|_1 + K(t) \|1_m^T x(t)\|,
\end{split}
\end{equation}
where $K(t) = \beta(t)/m_{\zeta}$.
From \eqref{barx} we find the following estimate
\begin{equation*}
\|1_m^T x(t)\| \leq \|x(0)\|_1 + \sum_{s=0}^{t-1}  \|\varepsilon (s)\|_1.
\end{equation*}
Combining this with \eqref{eq-3-50} and using \eqref{eq100}, we obtain
\begin{equation}\label{eq-3-52}
\begin{split}
&\bigg\| [A(t)x(t)]_i -y_i (t+1) \frac{1_m^T x(t)}{m} \bigg\|
\\
&\leq  C_0   \lambda^t \|x(0)\|_1 + \sum_{s=0}^{t-1} C_0  \lambda^{t-s-1} \|\varepsilon (s)\|_1
 \\
 &\quad+ K(t) \Big( \|x(0)\|_1 + \sum_{s=0}^{t-1} \|\varepsilon (s)\|_1\Big)
 \\
&\leq \Big( C_0  \lambda^t + K(t)\Big) \|x(0)\|_1 
\\
&\quad + m\sum_{s=0}^{t-1} \Big[ C_0 \lambda^{t-s-1} + K(t)\Big] \Big( \alpha (s+1) D +  \tau (s)\Big).
\end{split}
\end{equation}
By applying Lemma \ref{lowerbound}, \eqref{error_estimate} and the above inequality to the norm of \eqref{eq-3-51}, we obtain
\begin{align*}
&\|\hat{z}_i(t+1)-B_{\zeta}\bar{x}(t)\| 
\\
&\leq \frac{1}{\delta}\bigg(\Big( C_0  \lambda^t +K(t)\Big) \|x(0)\|_1 
\\
&\quad + m\sum_{s=0}^{t-1} \Big[ C_0 \lambda^{t-s-1}+K(t) \Big] \Big( \alpha (s+1) D +  \tau (s)\Big)\bigg)\\
&\qquad +\frac{1}{\delta}\bigg( d_i(t)\tau(t)\bigg).
\end{align*}
It remains to estimate the case $t=0$. By the algorithm, we have
\begin{equation*}
\begin{split}
\hat{z}_i (1) - \bar{x}(0) & = \frac{\hat{w}_i (1)}{y_i (1)} - \bar{x}(0) = \frac{\sum_{j=1}^m a_{ij} (0) x_j (0)}{\sum_{j=1}^m a_{ij}(0)} - \bar{x}(0).
\end{split}
\end{equation*}
Using this we find
\begin{equation*}
\begin{split}
\|z_i (1) -\bar{x}(0)\| &\leq \frac{1}{\delta} \|x(0)\|_1 + \frac{1}{m} \|x(0)\|_1
\\
& \leq 2 \|x(0)\|_1 \leq  \frac{2C_0  }{\delta} \|x(0)\|.
\end{split}
\end{equation*}
The proof is finished.
\end{proof}

By utilizing Proposition \ref{mean}, we analyze the relation between $\hat{z}_i(t)$ and $B_{\zeta}\bar{x}(t)$ under the assumptions on $\{\alpha (t)\}_{t \in \mathbb{N}}$, $\{\tau(t)\}_{t \in \mathbb{N}}$ and $\{\zeta(t)\}_{t \in \mathbb{N}}$ of the main theorems. To do this, we first recall a useful lemma from \cite{R Ned V - stochastic}.
\begin{lem}[\cite{R Ned V - stochastic}, $Lemma$ 3.1]\label{lemma5}
If $\lim_{k\rightarrow \infty}\gamma_k=\gamma$ and $0<\beta<1$, then
$$
\lim_{k\rightarrow \infty} \sum^k_{l=0}\beta^{k-l}\gamma_l = \frac{\gamma}{1-\beta}
$$
\end{lem}
\begin{cor}\label{convergent}
Suppose that the triggering function $\tau$ satisfies Assumption \ref{event-triggered}. Also, assume that the stepsize $\alpha (t)$ satisfies Assumption \ref{stepsize} or $\alpha (t) =1/\sqrt{t}$. Then we have
$$
\lim_{t\rightarrow \infty} \|\hat{z}_i(t+1)-B_{\zeta}\bar{x}(t)\| = 0 \ \text{for all $i$}.
$$   
\end{cor}

\begin{proof}
We recall from Proposition \ref{mean} the following inequality
\begin{equation}\label{eq-3-20}
\begin{split}
&\|\hat{z}_i(t+1)-B_{\zeta}\bar{x}(t)\|\\
&\leq \frac{1}{\delta}\bigg(\Big( C_0  \lambda^t +K(t)\Big) \|x(0)\|_1
\\
&\quad  + m \sum_{s=0}^{t-1} \Big[ C_0 \lambda^{t-s-1} +K(t)\Big] \Big( \alpha (s+1) D +  \tau (s)\Big)\bigg) \\
&\qquad +\frac{1}{\delta}\bigg(d_i(t)\tau(t)\bigg).
\end{split}
\end{equation} 
We notice that $\lim_{t \rightarrow \infty} t^{3/2} K(t) = 0$ by Lemma \ref{lem-3-4}. From this and the boundedness of $\alpha (s)$ and $\tau (s)$, it easily follows that
\begin{equation*}
\lim_{t \rightarrow \infty} \frac{1}{\delta}K(t) \|x(0)\|_1 + m \sum_{s=0}^{t-1} K(t) \Big(\alpha (s+1) D + \tau (s)\Big) = 0.
\end{equation*}
In addition, by Assumptions \ref{stepsize}, \ref{event-triggered} and \ref{y-trigger}, we know that $
\lim_{s\rightarrow \infty}\alpha(s+1) =0$, $\lim_{s\rightarrow \infty} \tau(s)=0
$ and $\lim_{s\rightarrow \infty} \zeta(s)=0$. Using this fact with Lemma \ref{lemma5} in the right hand side of \eqref{eq-3-20}, we deduce
$$
\lim_{t\rightarrow \infty} \|\hat{z}_i(t+1)-B_{\zeta}\bar{x}(t)\| = 0,
$$ 
which completes the proof.
\end{proof}
\begin{cor}\label{specific} If Assumptions 2.4 holds, and the stepsize is chosen as $\alpha(t)= \frac{1}{\sqrt{t}}$, then we have
\begin{align*}
&\sum^T_{t=0}\alpha(t+1)\|\hat{z}_i(t+1)-B_{\zeta}\bar{x}(t)\|\\
& \leq  \frac{C_0 }{\delta(1-\lambda)}  \|x(0)\|_1  + \frac{4mC_0  E_{\tau}(T)}{\delta(1-\lambda)}  +  \frac{C_0mD}{\delta(1-\lambda)} (1+\ln(T))
\\
&~ + \frac{1}{\delta}\sum_{t=0}^T K(t) \alpha (t+1) \Big[ \|x(0)\|_1 + \sum_{s=0}^{t-1} (\alpha (s+1) D + \tau (s))\Big].
\end{align*}
\end{cor}
\begin{proof}
By Proposition \ref{mean}, we have
\begin{equation}\label{eq-3-1}
\begin{split}
&\delta \sum_{t=0}^T \alpha (t+1) \|\hat{z}_i (t+1) -B_{\zeta} \bar{x}(t)\|
\\
&\leq \|x(0)\|_1 \sum_{t=0}^T (C_0 \lambda^t +K(t)) \alpha (t+1)
\\
&\quad + m \sum_{t=0}^T \alpha (t+1) \sum_{s=0}^{t-1} (C_0 \lambda^{t-s-1} +K(t)) (\alpha (s+1) D + \tau (s))
\\
&\quad  + \sum_{t=0}^T \alpha (t+1) (d_i (t) \tau (t)).
\end{split}
\end{equation}
The terms involving $K(t)$ are fit to the inequality of the lemma. Let us estimate each summation not involving $K(t)$ in the right hand side. Using that $\alpha (t) \leq 1$, the first term is bounded with   
\begin{equation}\label{eq-3-30}
\sum^T_{t=0} \alpha(t+1)\lambda^t \leq \sum^T_{t=0} {\lambda^t} \leq  \frac{1}{1-\lambda}.
\end{equation}     
The fourth term is bounded using
\begin{equation}\label{eq-3-31}
\sum_{t=0}^{T}\alpha (t+1) \frac{d_i(t)\tau (t)}{\delta} \leq \frac{m}{\delta} \sum_{t=0}^{T} \tau (t) = \frac{mE_{\tau}(T)}{\delta}.
\end{equation}
We estimate the second term using
\begin{equation}\label{eq-3-32}
\begin{split}
\sum^T_{t=0} \alpha(t+1)\sum^{t-1}_{s=0}\lambda^{t-s-1}\alpha(s+1) &= \sum^{T+1}_{t=1}\frac{1}{\sqrt{t}}\sum^{t}_{s=1}\lambda^{t-s} \frac{1}{\sqrt{s}}\\
&\leq \sum^{T+1}_{t=1}\sum^t_{s=1}\lambda^{t-s}\frac{1}{s}\\
&=\sum^{T+1}_{s=1}\frac{1}{s}\sum^{T+1}_{t=s}\lambda^{t-s}\\
&\leq \frac{1+\ln{(T+1)}}{1-\lambda}
\end{split}
\end{equation}
In order to estimate the third term, we estimate
\begin{equation*}
\begin{split}
&\sum_{s=0}^{t-1} \lambda^{t-s-1}\tau (s)
\\
 &= \sum_{s=0}^{[(t-1)/2]} \lambda^{t-s-1} \tau (s) + \sum_{s=[(t-1)/2]+1}^{t-1} \lambda^{t-s-1} \tau (s)
\\
& \leq \lambda^{(t-1)/2} \sum_{s=0}^{[(t-1)/2]} \tau (s) + \tau ([t/2] ) \sum_{s=[(t-1)/2]+1}^{t-1} \lambda^{t-s-1}
\\
&\leq \lambda^{(t-1)/2} E_{\tau} (T) + \frac{\tau ([t/2])}{1-\lambda},
\end{split}
\end{equation*}
where $[a]$ denotes the largest integer not larger than $a \in \mathbb{R}$.
Using this we derive
\begin{equation}\label{eq-3-33}
\begin{split}
&\sum_{t=1}^{T}\alpha (t+1) \Big[ \sum_{s=0}^{t-1} \lambda^{t-s-1} \tau (s)\Big]
\\
& \leq E_{\tau}(T) \sum_{t=1}^T \frac{\lambda^{(t-1)/2}}{\sqrt{t+1}}  +\frac{1}{1-\lambda}\sum_{t=1}^{T} \frac{\tau ([t/2])}{\sqrt{t+1}}
\\
& \leq \frac{E_{\tau}(T)}{1-\sqrt{\lambda}} +  \frac{E_{\tau}(T)}{1-\lambda} < \frac{3E_{\tau}(T)}{1-\lambda}.
\end{split}
\end{equation}
Putting the above estimates \eqref{eq-3-30}-\eqref{eq-3-33} in \eqref{eq-3-1}, we obtain
\begin{align*}
&\sum^T_{t=0}\alpha(t+1)\|\hat{z}_i(t+1)-B_{\zeta}\bar{x}(t)\|\\
& \leq  \frac{C_0 }{\delta(1-\lambda)} \|x(0)\|_1  + \frac{4mC_0  E_{\tau}(T)}{\delta(1-\lambda)}  +  \frac{C_0mD}{\delta(1-\lambda)} (1+\ln(T))
\\
&~  + \frac{1}{\delta}\sum_{t=0}^T K(t) \alpha (t+1) \Big[ \|x(0)\|_1 + \sum_{s=0}^{t-1} (\alpha (s+1) D + \tau (s))\Big].
\end{align*}
which finishes the proof.
\end{proof}

\section{Convergence estimates}


In this section we prove our main results, namely Theorems \ref{maintheorem} and \ref{case2_main}. In Section 3, we obtained the bound of the disagreement in agent estimates. Especially, Corollary \ref{convergent} and \ref{specific} investigate the difference between the state $\hat{z}_i(t)$ in the Algorithm \ref{algo} and $B_{\zeta}\bar{x}(t)$ in \eqref{eq-2-1a}. Based upon these results, Theorem \ref{maintheorem} and \ref{case2_main} can be proved by comparing the cost values computed at the points $B_{\zeta}\bar{x}(t)$ and $x^*$. 
\begin{lem}\label{submain}
Suppose Assumption \ref{GB} holds. Then for any $t \geq 0$ and $x \in \mathbb{R}^d$ we have
\begin{align}\label{combine}
&\sum_{i=1}^m \big(f_i(B_{\zeta}\bar{x}(t))-f_i(x)\big) \\
\nonumber&\leq \frac{m}{2\alpha(t+1)B_{\zeta}}(\|B_{\zeta}\bar{x}(t)-x\|^2-\|B_{\zeta}\bar{x}(t+1)-x\|^2)
\\
\nonumber&~+ \frac{B_{\zeta}m}{2\alpha(t+1)}\big(2\alpha(t+1)^2 D^2 + 2 \tau(t)^2\big)\\
\nonumber&~ + \frac{m}{\alpha(t+1)}\|B_{\zeta}\bar{x}(t)-x\|\tau(t)+ 2D\sum_{i=1}^{m} \|\hat{z}_i (t+1) -B_{\zeta}\bar{x}(t)\|.
\end{align}
\end{lem}
\begin{proof}
By convexity, we have
\begin{align*}
&f_i(\hat{z}_i(t+1))
\\
 &\leq f_i(x) + (\hat{z}_i(t+1)-x)\nabla f_i(\hat{z}_i(t+1))\\
&= f_i(x) + (B_{\zeta}\bar{x}(t)-x)\nabla f_i(\hat{z}_i(t+1))  
\\
&\quad + (\hat{z}_i(t+1)-B_{\zeta}\bar{x}(t))\nabla f_i(\hat{z}_i(t+1))\\
&= f_i(x) + \frac{1}{\alpha(t+1)} (B_{\zeta}\bar{x}(t)-x)(w_i(t+1)-x_i(t+1) 
\\
&\quad +e_i(t+1))+ (\hat{z}_i(t+1)-B_{\zeta}\bar{x}(t))\nabla f_i(\hat{z}_i(t+1)),
\end{align*}
where \eqref{eq-3-10} is used in the last equality.
Summing up the above inequality from $i=1$ to $i=m$, we find that
\begin{align*}
&\sum_{i=1}^m f_i(\hat{z}_i(t+1))-f_i(x) 
\\
&\leq \underbrace{\frac{m}{\alpha(t+1)}(B_{\zeta}\bar{x}(t)-x)(\bar{x}(t)-\bar{x}(t+1))}_{\uppercase\expandafter{\romannumeral1}}\\
&\quad + \underbrace{\frac{1}{\alpha(t+1)}(B_{\zeta}\bar{x}(t)-x)\sum^m_{i=1}e_i(t+1)}_{\uppercase\expandafter{\romannumeral2}} \\
&\quad + \underbrace{\sum^m_{i=1}(\hat{z}_i(t+1)-B_{\zeta}\bar{x}(t))\nabla f_i(\hat{z}_i(t+1))}_{\uppercase\expandafter{\romannumeral3}}.
\end{align*}
Now we estimate each term in the right hand side. First using the equality $\langle a,b\rangle=\frac{1}{2}(\|a\|^2+\|b\|^2-\|a-b\|^2)$ for $a,b \in \mathbb{R}^d$, we have
\begin{equation*}
\begin{split}
\uppercase\expandafter{\romannumeral1}& = \frac{m}{2\alpha(t+1)B_{\zeta}}(\|B_{\zeta}\bar{x}(t)-x\|^2-\|B_{\zeta}\bar{x}(t+1)-x\|^2
\\
&\quad \qquad+\|B_{\zeta}\bar{x}(t+1)-B_{\zeta}\bar{x}(t)\|^2).
\end{split}
\end{equation*}
Using \eqref{eq-2-1a} along with \eqref{error_estimate} and \eqref{eq-2-1}, we estimate the right-most term as 
\begin{align*}
&\|\bar{x}(t+1)-\bar{x}(t)\|^2 
\\
&\leq 2\bigg\|\frac{\alpha(t+1)}{m}\sum^m_{i=1}\nabla f_i(\hat{z}_i(t+1)) \bigg\|^2 + 2\bigg\|\frac{1}{m}\sum^m_{i=1}e_i(t+1) \bigg\|^2\\
&\leq 2 \alpha(t+1)^2 D^2  +2\tau(t)^2.
\end{align*}
We apply \eqref{error_estimate} again to  estimate
\begin{align*}
\uppercase\expandafter{\romannumeral2} \leq \frac{m}{\alpha(t+1)}\|B_{\zeta}\bar{x}(t)-x\|\tau(t),
\end{align*}
and use \eqref{eq-2-1} to deduce
\begin{equation*}
\uppercase\expandafter{\romannumeral3} \leq D\sum_{i=1}^{m} \|\hat{z}_i (t+1) -B_{\zeta}\bar{x}(t)\|.
\end{equation*}
Combining the above estimates on $\uppercase\expandafter{\romannumeral1}$,$\uppercase\expandafter{\romannumeral2}$ and $\uppercase\expandafter{\romannumeral3}$, we have
\begin{equation}\label{eq-4-3}
\begin{split}
&\sum_{i=1}^m \big(f_i(\hat{z}_i(t+1))-f_i(x)\big) \\
\nonumber&\leq \frac{m}{2\alpha(t+1)B_{\zeta}}(\|B_{\zeta}\bar{x}(t)-x\|^2-\|B_{\zeta}\bar{x}(t+1)-x\|^2)
\\
\nonumber&\quad + \frac{mB_{\zeta}}{2\alpha(t+1)}\big(2\alpha(t+1)^2 D^2 + 2 \tau(t)^2\big)\\
\nonumber&\quad + \frac{m}{\alpha(t+1)}\|B_{\zeta}\bar{x}(t)-x\|\tau(t)+ D\sum_{i=1}^{m} \|\hat{z}_i (t+1) -B_{\zeta}\bar{x}(t)\|.
\end{split}
\end{equation}
Finally we observe that \eqref{eq-2-1} gives us the estimate
\begin{equation*}
\sum_{i=1}^m \Big( f_i (B_{\zeta}\bar{x}(t)) -f_i  (\hat{z}_i (t+1))\Big) \leq D\sum_{i=1}^m \|B_{\zeta}\bar{x}(t) - \hat{z}_i (t+1)\|.
\end{equation*}
Summing up the above two inequalities, we obtain the desired estimate. 
\end{proof}
\subsection{Proof of Theorem \ref{maintheorem}}
We recall the following lemma for proving Theorem \ref{maintheorem}.
\begin{lem}[\cite{Ned Push sum} Lemma 7]\label{lemma7}
Consider a minimization problem $\min_{x\in\mathbb{R}^d}f(x)$, where $f:\mathbb{R}^d\rightarrow \mathbb{R}$ is a continuous function. Assume that the solution $X^*$ of the problem is nonempty. Let $\{x(t)\}$ be a sequence such that for all $x\in X^*$ and for all $t\geq 0$,
$$
\|x(t+1)-x\|^2\leq (1+b(t))\|x(t) - x \|^2- a(t)(f(x(t))-f(x))+c(t)
$$
where $b(t)\geq 0$, $a(t)\geq 0$ and $c(t)\geq 0$ for all $t\geq 0$ with $\sum^\infty_{t=0}b(t) < \infty$,$\sum^\infty_{t=0}a(t) = \infty$ and $\sum^\infty_{t=0}c(t)<\infty$. Then the sequence $\{x(t)\}$ converges to some solution $x^*\in X^*$ 
\end{lem}
By manipulating the estimate in Lemma \ref{submain}, we obtain the following estimate which is suitable for applying Lemma \ref{lemma7}.
\begin{cor}\label{main}
Under the same assumptions in Lemma \ref{submain} we have
\begin{equation*}
\begin{split}
&\|B_{\zeta}\bar{x}(t+1) -x\|^2 \leq (1+\tau(t))\|B_{\zeta}\bar{x}(t) -x\|^2 
\\
&\quad -\frac{2\alpha(t+1)B_{\zeta}}{m} (f(B_{\zeta}\bar{x}(t)) - f(x)) + c(t)+d(t),
\end{split}
\end{equation*}
where
\begin{equation*}
c(t)  =  \Big[2\alpha(t+1)^2 D^2 + 2 \tau(t)^2 +  \tau(t)\Big]B_{\zeta},
\end{equation*}
and
\begin{equation*}
d(t) = \frac{4\alpha (t+1)B_{\zeta}D}{m}\sum_{i=1}^m \|\hat{z}_i (t+1) -\bar{x}(t)\|.
\end{equation*}
\end{cor}
\begin{proof}
We use Young's inequality to find
\begin{equation*}
\|B_{\zeta}\bar{x}(t)-x\|\tau(t) \leq \frac{\|B_{\zeta}\bar{x}(t)-x\|^2 \,\tau(t) }{2B_{\zeta}} + \frac{\tau(t)B_{\zeta}}{2}.
\end{equation*}
Applying this to \eqref{combine}, we get
\begin{align}\label{eq-3-2}
&\quad\sum_{i=1}^m \big(f_i(B_{\zeta}\bar{x}(t))-f_i(x)\big) \\
\nonumber&\leq \frac{m}{2\alpha(t+1)B_{\zeta}}(1+\tau(t))\|B_{\zeta}\bar{x}(t)-x\|^2
\\
\nonumber &\quad -\frac{m}{2\alpha (t+1)B_{\zeta}}\|B_{\zeta}\bar{x}(t+1)-x\|^2 \\
\nonumber&\quad + \frac{mB_{\zeta}}{2\alpha(t+1)}\big(2\alpha(t+1)^2 D^2 + 2 \tau(t)^2\big)+ \frac{B_{\zeta}m \tau(t)}{2\alpha(t+1)} 
\\
\nonumber&\quad + 2D \sum_{i=1}^m \|\hat{z}_i (t+1) - B_{\zeta}\bar{x}(t)\|.
\end{align}
Dividing both sides by $\frac{m}{2\alpha (t+1) B_{\zeta}}$, it follows that
\begin{align}\label{eq-3-4}
&\frac{2\alpha (t+1)B_{\zeta}}{m}\sum_{i=1}^m \big(f_i(B_{\zeta}\bar{x}(t))-f_i(x)\big) \\
\nonumber&\leq  (1+\tau(t))\|B_{\zeta}\bar{x}(t)-x\|^2- \|B_{\zeta}\bar{x}(t+1)-x\|^2 \\
\nonumber&\quad +  B_{\zeta}^2\Big[ 2\alpha(t+1)^2 D^2 + 2 \tau(t)^2 + \tau(t) \Big] 
\\
\nonumber&\quad +  \frac{4\alpha (t+1)DB_{\zeta}}{m} \sum_{i=1}^m \|\hat{z}_i (t+1) - B_{\zeta}\bar{x}(t)\|. 
\end{align}
Rearranging this we obtain the desired estimate.
\end{proof}
Now we are ready to prove Theorem \ref{maintheorem}.
\begin{proof}[Proof of Theorem \ref{maintheorem}]
By Lemmas \ref{lemma7} and Corollary \ref{main} it is enough to prove $\sum_{t=1}^{\infty}(c(t) +d(t)) <\infty$, where
\begin{equation*}
c(t)  =  2\alpha(t+1)^2 D^2 + 2 \tau(t)^2 +  \tau(t),
\end{equation*}
and
\begin{equation*}
d(t) = \frac{4\alpha (t+1)DB_{\zeta}}{m}\sum_{i=1}^m \|\hat{z}_i (t+1) -B_{\zeta}\bar{x}(t)\|.
\end{equation*}
It follows that  $\sum^\infty_{t=0}c(t)<\infty$ by Assumptions \ref{stepsize} and \ref{event-triggered}.

Next we will show that $\sum^\infty_{t=0} d(t) <\infty.$ By Proposition \ref{mean}, it suffices to show that
\begin{equation}\label{eq-4-10}
\sum^\infty_{t=1}\alpha(t+1)\bigg( \lambda^t  +\tau(t) + \sum^{t-1}_{s=0}\lambda^{t-s-1}\alpha(s+1) + \sum^{t-1}_{s=0}\lambda^{t-s-1}\tau(s) \bigg)<\infty,
\end{equation}
and
\begin{equation*}
\sum_{t=1}^{\infty}\alpha (t+1)\Big[K(t)+\sum_{s=0}^{t-1} K(t) \Big(\alpha (s+1) D + \tau (s)\Big)\Big]<\infty.
\end{equation*}
The latter one is proved  in Lemma \ref{lem-4-4} below. We proceed to prove \eqref{eq-4-10}.
Using the Cauchy-Schwarz inequality, we have
\begin{align*}
\sum^\infty_{t=1}\alpha(t+1)\lambda^t \leq \frac{1}{2}\sum^\infty_{t=1}\alpha(t+1)^2 + \frac{1}{2}\sum^\infty_{t=1}\lambda^{2t} <\infty.
\end{align*}
By rearranging and using the decreasing property of $\alpha (t)$ in Assumption \ref{stepsize}, we find
\begin{equation*}
\begin{split}
&\sum^{\infty}_{t=1}\alpha(t+1)\sum^{t-1}_{s=0} \lambda^{t-s-1}\alpha(s+1) 
\\
& = \sum_{s=0}^{\infty} \sum_{t=s+1}^{\infty} \lambda^{t-s-1} \alpha (t+1) \alpha (s+1)
\\
&\leq \sum_{s=0}^{\infty} \Big( \sum_{t=s+1}^{\infty} \lambda^{t-s-1} \Big) \alpha (s+1)^2
\\
& = \frac{1}{1-\lambda} \sum_{s=0}^{\infty} \alpha (s+1)^2 < \infty.
\end{split}
\end{equation*}
Similarly, due to Assumption \ref{event-triggered}, the last term is bounded as
\begin{equation*}
\begin{split}
\sum^\infty_{t=0}\alpha(t+1)\sum^{t-1}_{s=0}\lambda^{t-s-1}\tau(s) & = \sum_{s=0}^{\infty}\sum_{s<t} \lambda^{t-s-1} \alpha (t+1) \tau (s)
\\
&\leq \sum_{s=0}^{\infty} \sum_{s<t} \lambda^{t-s-1} \alpha (s+1) \tau (s) 
\\
& = \frac{1}{1-\lambda} \sum_{s=0}^{\infty} \alpha (s+1) \tau (s) <\infty.
\end{split}
\end{equation*}
Gathering the above estimates, we find that $\sum_{t=1}^{\infty}(c(t) + d(t)) <\infty$. Hence by Lemma \ref{lemma7}, the sequence $\{B_{\zeta}\bar{x}(t)\}$ converges to some solution $x^*\in X^*$. Finally, we apply Corollary \ref{convergent} to conclude  that each sequence $\{\hat{z}_i(t)\}$, $i=1,\cdots,n,$ converges to the same solution $x^*$. The proof is done.    
\end{proof}
\begin{lem}\label{lem-4-4}Suppose that Assumption \ref{y-trigger} and Assumption \ref{event-triggered} hold. Then for the stepsize $\{\alpha(t+1)\}_{t\geq 0}$ satisfying Assumption \ref{stepsize} or $\alpha (t) =1/\sqrt{t}$, we have
\begin{equation*}
\sum_{t=1}^{\infty}\alpha (t+1) \Big[ K(t) + \sum_{s=0}^{t-1} K(t) \Big( \alpha (s+1)  + \tau (s)\Big) < \infty.
\end{equation*} 
\end{lem}
\begin{proof}
From Lemma \ref{lem-3-4}, we know that $\lim_{t \rightarrow \infty} t^{3/2} K(t) = 0$. Using this fact the summability of $\tau (s)$, it easily follows that
\begin{equation*}
\sum_{t=1}^{\infty} \alpha (t+1)\Big[ K(t) + \Big( \sum_{s=0}^{t-1} \tau (s)\Big) K(t)\Big] < \infty.
\end{equation*}
Next, for $\alpha (t)$ satisfying Assumption \ref{stepsize}, we observe that
\begin{equation*}
\begin{split}
\sum_{t=1}^{\infty}\alpha (t+1) K(t) \sum_{s=0}^{t-1} \alpha (s+1) & \leq \sum_{t=1}^{\infty} K(t) \sum_{s=0}^{t-1} \alpha (s+1)^2 <\infty.
\end{split}
\end{equation*}
For $\alpha (t) = 1/\sqrt{t}$, by using that $\sum_{s=0}^{t-1} 1/\sqrt{s+1} \leq 2 \sqrt{t}$, we deduce
\begin{equation*}
\begin{split}
\sum_{t=1}^{\infty}\alpha (t+1) K(t) \sum_{s=0}^{t-1} \alpha (s+1) & \leq 2\sum_{t=1}^{\infty} K(t) <\infty.
\end{split}
\end{equation*}
The proof is done.
\end{proof}

\subsection{Proof of Theorem \ref{case2_main}}
We now turn to the proof of Theorem \ref{case2_main}. we recall that
 \begin{equation*}
 H(-1) = 1\quad \textrm{and}\quad H(t) = \prod_{k=0}^{t} (1+\tau(k))\quad \textrm{for}\quad t\geq 1
 \end{equation*}
 and
 \begin{equation*}
 S(0) = 0\quad \textrm{and}\quad S(t) =\sum_{s=0}^{t-1} \frac{\alpha (s+1)}{H(s)}\quad \textrm{for}\quad t\geq 1.
 \end{equation*}
 
 First we find the boundedness of $H(t)$ and $S(t)$.
\begin{lem}\label{s(t+1)}
Let $\alpha(t)=\frac{1}{\sqrt{t}}$ and Assumption 2.4 hold. Then we have
\begin{equation}\label{eq-4-13}
\sup_{t\geq 0} H(t) < e^{E_{\tau}}\quad \textrm{and}\quad S(t)\geq  e^{-E_{\tau}}\sqrt{t},
\end{equation}
where $E_{\tau} = \sum_{t=0}^{\infty}\tau (t) <\infty$.
\end{lem}
\begin{proof}
By applying the inequality $1+x \leq e^x$ for $x \geq 0$ we estimate $H(t)$ as 
\begin{equation*}
\sup_{t\geq 0} H(t) <\exp{(\sum^\infty_{t=0}\tau(t))}=e^{E_{\tau}}.
\end{equation*}
Using this inequality, we deduce the following estimate
\begin{equation*}
\begin{split}
S(t)& = \sum_{s=0}^{t-1} \frac{\alpha (s+1)}{H(s)} \geq e^{-E_{\tau}} \sum_{s=1}^{t} \frac{1}{\sqrt{s}} \geq e^{-E_{\tau}} \sqrt{t}\quad \forall~ t \geq 1.
\end{split}
\end{equation*}
The proof is done.
\end{proof}
To prove Theorem \ref{case2_main}, we first modify Lemma \ref{submain} which states the boundedness of $\sum_{i=1}^m \big(f_i(\bar{x}(t))-f_i(x)\big)$, by replacing $\bar{x}(t)$ to $\frac{\sum^T_{t=0}\frac{\alpha(t+1)}{H(t)}B_\zeta\bar{x}(t)}{S(T+1)}$.
\begin{lem}\label{case2_submain}
Suppose that all the conditions are same as in Theorem \ref{case2_main}. Then we have
\begin{equation*}
\begin{split}
&f\bigg(\frac{\sum^T_{t=0}\frac{\alpha(t+1)}{H(t)}B_{\zeta}\bar{x}(t)}{S(T+1)}\bigg) - f(x) 
\\
&\leq \frac{me^{E_{\tau}}}{2\sqrt{T+1}}J_1 (T)+ \frac{2mDe^{E_{\tau}}}{\delta \sqrt{T+1}}J_2 (T) + \frac{2mDe^{E_\tau}}{ \delta\sqrt{T+1}}J_3 (T)
\end{split}
\end{equation*}
for any $T \in \mathbb{N}$, where $J_1 (T), J_2 (T)$, and $J_3 (T)$ are defined in Theorem \ref{case2_main}
\end{lem}
\begin{proof}
We recall from \eqref{eq-3-2} the following inequality
\begin{align}\label{eq-3-2}
&\quad\sum_{i=1}^m \big(f_i(B_{\zeta}\bar{x}(t))-f_i(x)\big) \\
\nonumber&\leq \frac{m}{2\alpha(t+1)B_{\zeta}}(1+\tau(t))\|B_{\zeta}\bar{x}(t)-x\|^2
\\
\nonumber&\quad -\frac{m}{2\alpha (t+1)B_{\zeta}}\|B_{\zeta}\bar{x}(t+1)-x\|^2 \\
\nonumber&\quad + \frac{mB_{\zeta}}{2\alpha(t+1)}\big(2\alpha(t+1)^2 D^2 + 2 \tau(t)^2\big)
\\
\nonumber &\quad+ \frac{B_{\zeta}m \tau(t)}{2\alpha(t+1)}  + 2D \sum_{i=1}^m \|\hat{z}_i (t+1) - B_{\zeta}\bar{x}(t)\|.
\end{align}
Dividing both sides of \eqref{eq-3-2} by $\frac{mH(t)}{2\alpha(t+1)}$, we get
\begin{align}\label{eq-3-3}
&\frac{2\alpha (t+1)}{mH(t)}\sum_{i=1}^m \big(f_i(B_{\zeta}\bar{x}(t))-f_i(x)\big) \\
\nonumber&\leq  \frac{\|B_{\zeta}\bar{x}(t)-x\|^2}{B_{\zeta}H(t-1)}- \frac{\|B_{\zeta}\bar{x}(t+1)-x\|^2}{B_{\zeta}H(t)} \\
\nonumber&\quad +  \frac{B_{\zeta}}{H(t)}\bigg(2\alpha(t+1)^2 D^2 + 2 \tau(t)^2+\tau(t)\Bigg) 
\\
\nonumber&\quad +  \frac{4\alpha (t+1)D}{mH(t)}  \sum_{i=1}^m \|\hat{z}_i (t+1) -B_{\zeta} \bar{x}(t)\|. 
\end{align}
Summing this from $t=0$ to $t=T$ we obtain
\begin{align}\label{eq415}
&\sum_{t=0}^{T}\Bigg[\frac{2\alpha (t+1)}{mH(t)}\sum_{i=1}^m \big(f_i(B_{\zeta}\bar{x}(t))-f_i(x)\big) \Bigg]\\
\nonumber&\leq  \frac{\|\bar{x}(0)-x\|^2}{B_{\zeta}H(-1)}- \frac{\|\bar{x}(T+1)-x\|^2}{B_{\zeta}H(T)} \\
\nonumber&\quad +  \sum_{t=1}^T \frac{B_{\zeta}}{H(t)}\bigg(2\alpha(t+1)^2 D^2 + 2 \tau(t)^2+\tau(t)\Bigg) 
\\
\nonumber &\quad + \sum_{t=1}^T \frac{4\alpha (t+1)D}{mH(t)}  \sum_{i=1}^m \|\hat{z}_i (t+1) - B_{\zeta}\bar{x}(t)\|.
\end{align}
This, together with the fact that $H(-1)=1$ and $H(t) \geq 1$, gives
\begin{equation}\label{eq-3-40}
\begin{split}
&\sum_{t=0}^{T}\Bigg[\frac{2\alpha (t+1)}{mH(t)}\sum_{i=1}^m \big(f_i(B_{\zeta}\bar{x}(t))-f_i(x)\big) \Bigg]\\
 &\leq  \frac{\|\bar{x}(0)-x\|^2}{B_{\zeta}}   +  \sum_{t=0}^T  \bigg(2\alpha(t+1)^2 D^2 + 2 \tau(t)^2+\tau(t) \bigg)B_{\zeta}
\\
&\quad   +  \sum_{t=0}^T  \bigg(  \frac{4\alpha (t+1)D}{m}  \sum_{i=1}^m \|\hat{z}_i (t+1) -B_{\zeta} \bar{x}(t)\|\bigg).
\end{split}
\end{equation}
We find
\begin{equation}\label{eq-3-41}
  \sum^T_{t=0} 2\alpha(t+1)^2 D^2 = 2D^2 \sum^{T+1}_{t=1}\frac{1}{t}\leq 2D^2 \Big(1+ \ln{(T+1)}\Big),
\end{equation}
and by definition \eqref{eq-2-8} we have
\begin{equation}\label{eq-3-42}
 \sum^T_{t=0} \Big(2\tau(t)^2 +\tau (t)\Big) = 2E_{\tau,2}(T)+E_{\tau}(T).
\end{equation}
Finally we estimate the last term of the right hand side of \eqref{eq-3-40} using Corollary \ref{specific} as follows:
\begin{equation*}
\begin{split}
&\sum_{t=0}^T  \bigg(  \frac{4\alpha (t+1)D}{m}  \sum_{i=1}^m \|\hat{z}_i (t+1) - B_{\zeta}\bar{x}(t)\|\bigg) \\
&\medmath{\leq4D \Big( \frac{C_0 }{\delta(1-\lambda)}  \|x(0)\|_1  + \frac{4mC_0  E_{\tau}(T)}{\delta(1-\lambda)}  +  \frac{C_0mD}{\delta(1-\lambda)} (1+\ln(T))\Big)}
\\
&\quad  \medmath{+\frac{4D}{\delta}  \sum_{t=0}^T K(t) \alpha (t+1) \Big[ \|x(0)\|_1 + \sum_{s=0}^{t-1} (\alpha (s+1) D + \tau (s))\Big].}
\end{split}
\end{equation*}
Putting this estimate, \eqref{eq-3-41} and \eqref{eq-3-42} in \eqref{eq-3-40}, we achieve the following estimate
\begin{align*} 
&\sum_{t=0}^{T}\Bigg[\frac{2\alpha (t+1)}{mH(t)}\sum_{i=1}^m \big(f_i(B_{\zeta}\bar{x}(t))-f_i(x)\big) \Bigg]\\
\nonumber&\leq   J_1 (T)+  \frac{4D}{\delta} J_2 (T) +\frac{4D}{\delta} J_3 (T).
\end{align*}
Now we set $S(T) = \sum_{t=0}^{T-1} \frac{\alpha (t+1)}{H(t)}$ for $T \in \mathbb{N}$ and divide the both sides by $\frac{2S(T+1)}{m}$. Then we apply the convexity of $f_i$ in the left hand side and use the lower bound $S(T+1) \geq e^{-E_{\tau}}{\sqrt{T+1}}$ to the right hand side, which leads to
\begin{equation*}
\begin{split}
&f\Bigg(\frac{\sum^T_{t=0}\frac{\alpha(t+1)}{H(t)}B_{\zeta}\bar{x}(t)}{S(T+1)}\Bigg) - f(x) 
\\
&\leq \frac{me^{E_{\tau}}}{2\sqrt{T+1}}J_1 (T)+ \frac{2mDe^{E_{\tau}}}{\delta \sqrt{T+1}}J_2 (T)  + \frac{2mDe^{E_\tau}}{ \delta\sqrt{T+1}}J_3 (T).
\end{split}
\end{equation*}
The proof is finished.
\end{proof}
Now we are ready to give the proof of Theorem \ref{case2_main}.
\begin{proof}[Proof of Theorem \ref{case2_main}]
Using the definition of $\tilde{z}_i$ and Assumption 2.1, we find
\begin{equation*}
\begin{split}
&f(\tilde{z}_i(T+1))-f\bigg(\frac{\sum^T_{t=0}\frac{\alpha(t+1)}{H(t)}B_{\zeta}\bar{x}(t)}{S(T+1)}\bigg) 
\\
& =f\bigg(\frac{\sum_{t=0}^T \frac{\alpha (t+1)}{H(t)} \hat{z}_i (t+1)}{S(T+1)} \bigg)-f\bigg(\frac{\sum^T_{t=0}\frac{\alpha(t+1)}{H(t)}B_{\zeta}\bar{x}(t)}{S(T+1)}\bigg) 
\\
&\leq \frac{mD}{S(T+1)}\sum^T_{t=0}\frac{\alpha(t+1)}{H(t)}\|\hat{z}_i(t+1)-B_{\zeta}\bar{x}(t)\|.
\end{split}
\end{equation*}
Then by Corollary \ref{specific} with the fact that $H(t) \geq 1$ and \eqref{eq-4-13}, we have
\begin{align*}
&f(\tilde{z}_i(T+1))-f\bigg(\frac{\sum^T_{t=0}\frac{\alpha(t+1)}{H(t)}B_{\zeta}\bar{x}(t)}{S(T+1)}\bigg) 
\\
 &\quad\leq \frac{mDe^{E_{\tau}}}{\delta \sqrt{T+1}}J_2 (T)
+ \frac{mDe^{E_{\tau}}}{\delta \sqrt{T+1}} J_3 (T).
\end{align*}
Combining this inequality with Lemma \ref{case2_submain}, we obtain
\begin{align*}
&f(\tilde{z}_i(T+1))-f(x^*) \\
&\leq \frac{me^{E_{\tau}}}{2\sqrt{T+1}}J_1 (T) + \frac{3mDe^{E_{\tau}}}{\delta \sqrt{T+1}}J_2 (T)  + \frac{3mDe^{E_\tau}}{ \delta\sqrt{T+1}}J_3 (T).
\end{align*}
which is the desired estimate. Moreover, we see that the right hand side is bounded by $O(\log (T+1)/ \sqrt{T+1})$ using Lemma \ref{lem-4-4}.
\end{proof}
  
\section{Simulations}

In this section, we present simulation results of the proposed event-triggered gradient-push method to demonstrate that the theoretical results can be realized in practice. We consider the decentralized least squares problem:
\begin{equation*}
\min_{x\in\mathbb{R}^d}~\sum_{i=1}^m f_i (x)\quad \textrm{with}\quad f_i (x) =\|q_i - p^T_ix\|^2,
\end{equation*} 
where, each agent $i$ in $\mathcal{V}= \{1,\cdots,m\}$ is given the local cost function $f_i$. The variable $p_i\in \mathbb{R}^{d\times p}$ is the input data and the variable $q_i\in\mathbb{R}^{p}$ is the output data. The data are generated according to the linear regression model $q_i = p_i^T \tilde{x} + \varepsilon_{i}$ where $\tilde{x}\in \mathbb{R}^d$ is the true weight vector, $p_i$'s are uniformly random values between 0 and 1, and $\varepsilon_i$'s are jointly Gaussian and zero mean. The initial points $x_i(0)$ are independent random variables, generated by a standard Gaussian distribution. In this simulation, we set the problem dimensions and the number of agents as $d=5 $, $p=1 $, and $m=50 $. We use connected directed graph where every node has four out-neighbors.

\medskip

\noindent \textbf{Test 1.} Here we fix $\alpha (t) = 1/t^{0.52}$ which satisfies Assumption \ref{stepsize} and $\zeta(t) = 1/(3t^3)$ which satisfies Assumption \ref{y-trigger}, and  consider various choices of $\tau (t)$. We  measure the relative distance between the states $z_i (t)$ and the optimal point $x^*$ the value 
\begin{equation*}
R_d (t) = \frac{\sum_{i=1}^m \|z_i (t) - x^*\|}{\sum_{i=1}^m \|z_i (0) - x^*\|},
\end{equation*}
We set the termination time $k_f$ as the first time $k \in \mathbb{N}$ when $R_d (k) < 10^{-2}$. And we let $N_x$ and $N_y$ be the average of total number of triggers for all agents until the termination time associated with $\tau(t)$ and $\zeta(t)$, respectively. Table \ref{known results} indicates the average of those values depending on $\tau(t)$ and $\zeta(t)$ in 100 trials.

\begin{table*}
\begin{tabularx}{\textwidth}{@{}l*{10}{C}c}
\toprule
$\tau(t)$    & 0 & 0 &$1/ t^{1.1}$ & $1/t^{1.1}$ & $1/t^{1.3}$ & $1/t^{1.3}$ & $1/t^{1.5}$ & $1/t^{1.5}$ & $1/t^{1.7}$ &$1/t^{1.7}$ \\ 
$\zeta(t)$  & 0      & $1/ (3t^3)$           &  0      & $1/ (3t^3)$   &  0      & $1/ (3t^3)$  &  0      & $1/ (3t^3)$      & 0      & $1/ (3t^3)$      \\ 
\midrule
$N_x$ & 11425        & 11305      & 3125   & 3152   & 3299    & 3288   & 8860  & 8767     & 15537   & 11177     \\ 
\addlinespace
$N_y$    & 11425     & 26          & 72705  & 32  & 15696   & 27  & 11644   & 26    & 15572    & 26    \\ 
\addlinespace
$k_f$      & 11425       & 11305          & 72705   & 72757 & 15696   & 15572 & 11644   & 11514    & 15572   & 11207      \\
\bottomrule
\end{tabularx}
\vspace{0.1cm}
\caption{The number of triggers and termination time depending on $\tau(t)$ and $\zeta(t)$}
\label{known results}
\end{table*}
We first look at the effect of $\zeta(t)$, the threshold  for variables $y_i(t)$. Table \ref{known results} shows that an existence of the threshold $(\zeta \neq 0)$ does not bring a big difference in the termination time if we compare the cases $\zeta (t) =0$ and $\zeta (t) = 1/(3t^3)$ with same $\tau (t)$, but there is a big improvement in the number of triggers for $y_i(t)$. Next we discuss the values $N_x$ and $k_f$ of Table \ref{known results} in terms of $\tau(t)$, the threshold for  variables $x_i(t)$. As in Table \ref{known results}, some cases give us similar or worse results compared to the cases $\tau(t)=0$. For $\tau(t)=1/t^{1.1}$, the number of triggers is decreased by more than $70\%$, and the termination time increased by more than $530\%$. For $\tau(t) = 1/t^{1.7}$, both the termination time and the number of triggers  increased by almost $36\%$ when $\zeta(t)=0$ and remain similar when $\zeta(t)=1/(3t^3)$ compared to the cases $\tau(t)=0$. For $\tau(t)=1/t^{1.5}$, the termination time is almost same, the number of triggers  decreased by more than $20\%$. These results show that the proposed gradient-push with event-triggered communication with proper $\tau(t)$ and $\zeta(t)$ can diminish the  number of communications to achieve the convergence  compared to the gradient-push algorithm without triggering. 
The threshold functions $\tau(t)$ and $\zeta(t)$ should be chosen carefully considering the characteristics of the given optimization problem.

\medskip 

\begin{figure}[!htbp]
\centering
   \includegraphics[width=7.4cm]{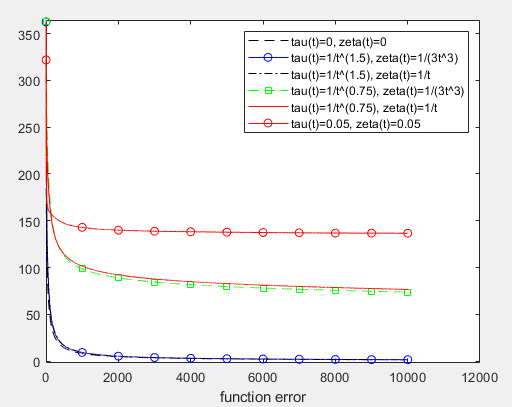} 
   \hfil
\caption{The values of $R_f (t)$ with different choices of $\tau (t)$.}
\label{figure1}
\end{figure}

\begin{figure}[!htbp]
\centering
 \includegraphics[width=7.4cm]{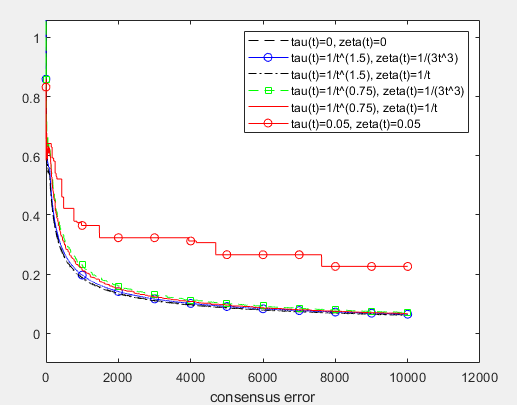}
   \hfil
\caption{The values of $R_c (t)$  with different choices of $\tau (t)$.}
\label{figure1}
\end{figure}

\noindent \textbf{Test 2.} Here we fix $\alpha (t) =1/\sqrt{t}$ and take several choices of $\tau (t)$ and $\zeta(t)$. We measure the relative cost error and the consensus error given by
\begin{equation*}
R_f (t) =  \sum_{i=1}^m ( f(\tilde{z}_i (t)) -f^*),\quad 
R_c (t) =  \max_{i,j \in V} \|z_i (t) - z_j (t)\|. 
\end{equation*}
For $\tau(t)$, we consider two cases where $\tau(t) = 1/t^{1.5}$ satisfying Assumption \ref{event-triggered} and $\tau(t) = 1/t^{0.75}$ not satisfying Assumption \ref{event-triggered}. And for $\zeta(t)$, we consider also consider two cases where $\zeta(t) = 1/(3t^{3})$ satisfying Assumption \ref{y-trigger} and $\zeta(t) = 1/t$ not satisfying Assumption \ref{y-trigger}. Additionally, we test two constant cases $\tau (t), \zeta(t) =0$ and $\tau (t),\zeta(t) =0.05$. Figure 1 depicts the graph of the values of $R_f (t)$ versus the iteration time. The result shows that the cost error decreases to zero when $\sum_{t=1}^{\infty} \tau (t) < \infty$  while it does not converge to zero when $\sum_{t=1}^{\infty} \tau (t) = \infty$. This agrees with the convergence result of Theorem \ref{case2_main}. Figure 2 illustrates the consensus error $R_c (t)$. The result shows that the consensus error decreases to zero for any choices $\tau (t)$ except the case $\tau (t) =0.05$. This numerical result supports the theoretical result obtained in Corollary \ref{convergent}.

\section{Conclusion}
In this work, we considered the gradient-push algorithm with event-triggered communication for distributed optimization problem whose agents are connected by directed graphs. We showed that by the algorithm each agent's state converges to a common minimizer under a diminishing and summability condition on the stepsize and the triggering function. Numerical simulations have been conducted to support the convergence results. It would be interesting further about how to choose an optimal triggering function for reducing the communication burden, reflecting various  elements such as the connectivity of graphs and the number of agents.



\end{document}